%% file: ms.tex
\pdfoutput=1

\documentclass[11pt, a4paper]{article}

\usepackage[draft]{sty/dl-en}
\usepackage{sty/config}

\include{macros}

\begin{document}

\maketitle

\footnotefirstpage

Given two sequences
\[
	X_0 \hookrightarrow \dots \hookrightarrow X_n
	\hookrightarrow \dots
	\qand
	Y_0 \hookrightarrow \dots \hookrightarrow
	Y_n \hookrightarrow \dots
\]
made of core-compact spaces with continuous injective
transition maps, we shall show that the
canonical continuous bijection
\[
	\textstyle
	\bigcup_{n < \OrdinalOmega} X_n \times Y_n
	\longrightarrow
	\left(\bigcup_{n < \OrdinalOmega} X_n\right) \times
	\left(\bigcup_{n < \OrdinalOmega} Y_n\right)
\]
is a homeomorphism.

By doing so, we shall sharpen a traditional tool used in the topology
of CW\=/complexes:
since their definition by Whitehead, one typical problem is
to show that a product of two CW-complexes is again a CW-complex. 
In 1956, Milnor showed that the product of two countable CW-complexes
is again a CW\=/complex
\cite{doi:10.2307/1969609}.
This result was then extended to products of locally countable
CW-complexes
\cite[5.2]{doi:10.1007/978-1-4684-6254-8}.
In 2015, someone on MathOverflow asked for
general commutation properties of sequential colimits
and fibre products.
Harpaz showed
that countable unions of closed embeddings between
separated and locally quasi-compact spaces commute with finite products
\cite{mathoverflow:215576}.
We extend Harpaz's result in several directions:
\begin{itemize}
	\item
		we remove the separation assumption;
	\item
		we replace local quasi-compactness
		with core-compactness;
	\item
		we remove the assumption that the transition maps
		be closed.
\end{itemize}

For the sake of the reader, we shall start by recalling the definition
and the elementary properties of core-compact spaces.

\section{Recollection on core-compact spaces}

The product with a topological space
does not always commute with colimits.
The class of spaces \( X \) for which the functor
\( Y \mapsto X \times Y \) commutes with all colimits was
characterised by Day and Kelly
\cite{doi:10.1017/s0305004100045850}.
These spaces were then called `core-compact' a few years after
\cite{doi:10.1090/s0002-9947-1978-0515540-7}.

\begin{definition}[(Way-below relation)]
	Let \( S, T \) be two subsets of a topological space
	\( X \).
	One writes \( S \ll T \) if for every open cover
	\( T \subset \cup_{i \in I} U_i \), there exists a finite
	subset \( J \subset I \) such that
	\( S \subset \cup_{i \in J} U_i \).
\end{definition}

\begin{definition}[(Core-compact space)]
	A topological space \( X \) is said to be core-compact if
	for every point \( x \in X \) and every open neighbourhood
	\( x \in U \), there exists an open
	neighbourhood \( V \) with \( x \in V \ll U \).
\end{definition}

The main example of core-compact spaces
are the locally quasi-compact ones:
if \( X \) is locally quasi-compact and \( x \in U \),
there is a quasi-compact
neighbourhood \( x \in K \subset U \), then letting \( V \) be the
interior of \( K \), one has \( x \in V \ll U \).
For sober spaces,
core-compactness coincides with local quasi-compactness
\cite[4.5]{doi:10.1090/s0002-9947-1978-0515540-7} and
there exists core-compact spaces which are not locally quasi-compact
\cite[\S 7]{doi:10.1090/s0002-9947-1978-0515540-7}.

We gather easy results regarding core-compact spaces and the
way-below relation that we shall use in the proof of the theorem:

\begin{itemize}
	\item
		a product of two core-compact spaces is core-compact;
	\item
		if \( S \ll T \) and \( T \subset Q \), then
		\( S \ll Q \);
	\item
		if \( S \ll T \) and \( R \subset S \), then
		\( R \ll T \);
	\item
		if \( X \) is core-compact and \( \mathfrak B(X) \) is
		a basis of opens of \( X \), then for every open
		\( W \subset X \), one has
		\[
			W = \bigcup_{\substack{U \ll W\\ U \in \mathfrak B(X)}} U
		\]
	\item
		if \( f \) is a continuous map
		and \( S \ll T \), then \( f(S) \ll f(T) \);
	\item
		if \( I \) is finite then
		\[
			\{S_i \ll T\}_{i \in I}
			\implies \cup_{i \in I} S_i \ll T
		\]
\end{itemize}

The following lemma is a variant of a theorem from Wallace
\cite[p. 142]{MR0070144}.

\begin{lemma}[(Product interpolation)]
	Let \( X \) and \( Y \) be two core-compact spaces.
	Let \( S \subset X \) and \( T \subset Y \) be subsets
	and let \( W \subset X \times Y \) be an open subset such that
	\( S \times T \ll W \).
	Then there exists
	\[
		S \times T \subset U_S \times V_T \ll W
	\]
	with \( U_S \) and \( V_T \)
	open neighbourhoods.
\end{lemma}

\begin{proof}
	Let \( x \in X \) and let \( \Theta_x \) be the set of opens
	\( U \times V \ll W \) with \( x \in U \).
	Let \( q \From X \times Y \to Y \) denote the projection on
	the second factor.
	Then \( \cup_{\Theta_x} V = q(W) \) and \( T \ll q(W) \).
	So there is a finite subset \( I_x \subset \Theta_x \) such that
	\( \{ x \} \times T \subset \cup_{I_x} U_i \times V_i
	\ll W \).
	Let \( U_x \coloneqq \cap_{I_x} U_i \) and
	let \( V_x \coloneqq \cup_{I_x} V_i \), then by construction
	\[
		\{ x \} \times T \subset U_x \times V_x \ll W
	\]

	Let \( p \From X \times Y \to X \) denote the projection on
	the first factor.
	And let \( \Gamma \) denote the collection
	of opens of the form \( U_x \times V_x \ll W \).
	One has
	\( S \ll p(W) \) and \( p(W) = \cup_\Gamma U_x \).
	So there is a finite subset \( I_S \subset \Gamma \) such that
	\( S \times T \subset \cup_{I_S} U_{x_i} \times V_{x_i} \ll W \).
	Let \( U_S \coloneqq \cup_{I_S} U_{x_i} \)
	and let \( V_T \coloneqq \cap_{I_S} V_{x_i} \), then
	by construction
	\[
		S \times T \subset U_S \times V_T \ll W
	\]
\end{proof}

\section{Theorem}

\begin{remark}[(Ascending chains of open neighbourhoods)]
	By defintion of the topology of a union,
	a subset \( U \subset \cup_{n < \OrdinalOmega} X_n \)
	is open if and only
	if \( U \cap X_n \subset X_n \) is open
	for every \( n < \OrdinalOmega \).

	Then one can build an open set of the union using
	an ascending chain of open neighbourhoods:
	find a sequence of open
	subsets \( U_n \subset X_n \) such that
	\( U_n \subset U_{n+1} \)
	for every \( n < \OrdinalOmega \).
	Then
	\[
		\textstyle
		U \coloneqq \bigcup_{n < \OrdinalOmega} U_n
	\]
	is open in \( \cup_{n < \OrdinalOmega} X_n \).
	Indeed, for every \( p < \OrdinalOmega \),
	\( U \cap X_p
	= \cup_{p \leq n < \OrdinalOmega} U_n \cap X_p \)
	is a union of open subsets of \( X_p \).
\end{remark}

\begin{theorem}
	Let
	\[
		X_0 \hookrightarrow \dots \hookrightarrow X_n
		\hookrightarrow \dots
		\qand
		Y_0 \hookrightarrow \dots \hookrightarrow
		Y_n \hookrightarrow \dots
	\]
	be two sequences of core-compact topological
	spaces with continuous injective transition maps,
	then the canonical continuous bijection
	\[
		\textstyle
		\bigcup_{n < \OrdinalOmega} X_n \times Y_n
		\longrightarrow
		\left(\bigcup_{n < \OrdinalOmega} X_n\right) \times
		\left(\bigcup_{n < \OrdinalOmega} Y_n\right)
	\]
	is a homeomorphism.
\end{theorem}

\begin{proof}
	Let \( W
	\subset \cup_{n < \OrdinalOmega} X_n \times Y_n \)
	be an open subset and
	write \( W_n \) for the intersection \( W \cap X_n \)
	for each \( n < \OrdinalOmega \).

	We need to show that for every \( (x, y) \in W \), one can
	find open neighbourhoods
	\( U \subset \cup_{n < \OrdinalOmega} X_n \)
	and \( V \subset \cup_{n < \OrdinalOmega} Y_n \)
	such that \( (x,y) \in U \times V \subset W \).
	For this, we shall build two ascending chains of open
	neighbourhoods such that for each \( n < \OrdinalOmega \),
	one has \( (x, y) \in U_n \times V_n \ll W_n \).

	Without loss of generality, one may assume that
	\( (x,y) \in X_0 \times Y_0\).
	Because \( X_0 \times Y_0 \) is core-compact one can find
	\( (x, y) \in U_0 \times V_0 \ll W_0 \).
	Assume that for \( n < \OrdinalOmega \), one has built
	open neighbourhoods \( U_n \subset X_n \)
	and \( V_n \subset Y_n \) with
	\( (x, y) \in U_n \times V_n \ll W_n \).
	Since \( W_n \subset W_{n+1} \),
	one has \( U_n \times V_n \ll W_{n+1} \)
	in \( X_{n+1} \times Y_{n+1} \).
	By the interpolation lemma, one can find two open neighbourhoods
	\( U_{n +1} \subset X_{n +1} \) and
	\( V_{n +1} \subset Y_{n+1} \) such that
	\( U_n \times V_n
	\subset U_{n+1} \times V_{n +1} \ll W_{n +1} \).

	By construction \( U \coloneqq \cup_{n < \OrdinalOmega} U_n \)
	and \( V \coloneqq \cup_{n < \OrdinalOmega} V_n \) are
	open neighbourhoods of the unions such that
	\( (x, y) \in U \times V \subset W \).
\end{proof}

\section{Counter-examples}

For the sake of completeness, we collect here two examples showing that
the assumptions of the theorem cannot be easily removed.

\subsection{Filtered unions}

One might want to replace the ordinal \( \OrdinalOmega \)
indexing the union by a filtered poset.
If it were possible the product of two CW-complexes would
always become a CW-complex.
Dowker showed that this is not the case by exhibiting a CW\=/complex
\( X \) with countably many cells and a CW-complex \( Y \) with
uncountably many cells such that \( X \times Y \) be not a
CW-complex
\cite{doi:10.2307/2372262}.

Further classification results have shown that for the product
of two CW\=/complexes to be a CW\=/complex, at least one of the two
must be locally countable
\cite[Thm. 1]{arXiv:1710.05296}.

\subsection{Non core-compact spaces}

One may ask for partial removal of the core-compactness assumption,
for example on one side of the product.

The following counter-example is due to Hamcke
\cite{mathstackexchange:1255678}.
Consider
a countable family of pointed circles \( C_0, \dots, C_n, \dots \).
Each finite wedge of circles \( \vee_{i \leq n} C_i \) is core-compact,
however the canonical continuous bijection
\[
	\textstyle
	\bigcup_n \Rationals \times \vee_{i \leq n} C_i
	\longrightarrow
	\Rationals \times \bigcup_n \vee_{i \leq n} C_i
\]
is not a homeomorphism, for the set \( A \) whose intersection with
\( \Rationals \times C_n \) is
\[
	A_n = A \cap (\Rationals \times C_n) \coloneqq
	\left\{ \left(r, \e^{2\ii\numberpi \theta}\right)
		\,\middle \vert\,
	\frac{\numberpi}{n} \leq r \leq \frac{\numberpi}{n}
	+ \max(\theta, 1-\theta) \right\}
\]
is closed in the finer topology but not closed in the product topology,
since \( (r=0,\theta=0) \) is a limit point.

\bibliography{ms}

\end{document}

%% file: macros.tex
\newcommand{\OrdinalOmega}{\omegaup}